\documentclass[leqno,a4paper]{amsart}

\usepackage{amsmath,amsfonts,amsthm,amssymb,dsfont}
\usepackage{relsize} %for \mathlarger - makes large sigma sum

\usepackage{	tgheros}
\usepackage{microtype}
\usepackage{color,soul}%highlight \hl
\usepackage[]{mdframed} % for frames
\usepackage{xcolor} %colourboxes
\usepackage{color}
\usepackage{soul} %for \st command to cross/strike text
\usepackage{cite}

\usepackage{mathrsfs}
\usepackage{enumerate, xspace}
\usepackage{graphicx}
\usepackage{color}
\usepackage[makeroom]{cancel} %cancellation sign

    \usepackage[all, knot]{xy}
    \xyoption{arc} 

\usepackage{pgf,tikz}\usepackage{mathrsfs}\usetikzlibrary{arrows}
\usepackage{enumitem} % enumerate (i) (ii) etc
\usepackage{tikz,lipsum,lmodern} %coloredboxed and more
\usepackage{tcolorbox} %coloredboxes
\usepackage{verbatim}

\usepackage{caption}
\usepackage{subcaption}
\usepackage[pdfauthor={Arman Darbinyan  and},pdftitle={project 2},pdfkeywords={Discrete Logarithm Problem, Post-quantum Cryptography, Infinite Groups, NP-hardness},pdftex]{hyperref}

\hypersetup{colorlinks,%
citecolor=black,%
filecolor=black,%
linkcolor=black,%
urlcolor=black,%
}

\theoremstyle{plain}

    \newtheorem*{conclusion*}{Conclusion}
				\newtheorem*{nttn*}{Notation}

                    \newtheorem{df}{Definition}[section]

                \newtheorem{conclusion}{Conclusion}[section]
				
				\newtheorem{thm}{Theorem}[section]
                \newtheorem{notation}{Notation}[section]

				\newtheorem{lem}[thm]{Lemma}		
						\newtheorem{cor}[thm]{Corollary}

\newtheorem{rem}[thm]{Remark}

	\theoremstyle{remark} 	

\DeclareMathOperator{\zz}{\mathbb{Z}}
\DeclareMathOperator{\nn}{\mathbb{N}}

\DeclareMathOperator{\calA}{\mathcal{A}}

\DeclareMathOperator{\calC}{\mathcal{C}}

\DeclareMathOperator{\calF}{\mathcal{F}}

\DeclareMathOperator{\calO}{\mathcal{O}}

  {\begin{center}
    \colorbox{#1}{\begin{minipage}{0.9\textwidth}}}
  {\end{minipage}\end{center}}

\title{Hard Instances of Discrete Logarithm Problem and Cryptographic Applications}
\date{\today}
\author{Christopher Battarbee, Arman Darbinyan, Delaram Kahrobaei}
\address{Sorbonne University, France}
\email{kit.battarbee@ed.ac.uk}
\address{University of Southampton, UK}
\email{a.darbinyan@soton.ac.uk}
\address{The City University of New York, Queens College and Graduate Center, USA}
\email{delaram.kahrobaei@qc.cuny.edu}

\begin{document}

\begin{abstract}
Let $f: \nn \rightarrow \nn$ be an arbitrary integer valued function. The goal of this note is to show that one can construct a finitely generated group in which the discrete log problem is polynomially equivalent to computing the function $f$. In particular, we provide infinite, but finitely generated groups, in which the discrete logarithm problem is arbitrarily hard. As a corollary, we construct a family of two-generated groups that have polynomial time word problem and NP-hard discrete log problem. Additionally, using our framework, we propose a generic scheme of cryptographic protocols, which might be of independent interest.

\end{abstract}

 \maketitle

%%%%%%%%%%%%%%%%%%%%%%%%%%%%%%%%%%%%%%%%%%%%%%%%%%%%%%%%%%%%%% 
\section{Introduction}

The Discrete Logarithm Problem (DLP) is the foundational problem of public-key cryptography. Its security underpins that of the Diffie-Hellman key exchange, which is, for example, the mechanism for establishing secrets in TLS 1.3. Despite its instrumental role in facilitating billions of everyday secure communications, the difficulty of solving the DLP remains a subject about which relatively little is known.

The classical result providing a lower bound on the complexity of DLP is that of Shoup, presented in \cite{shoup1997lower}: in something called the ``generic group model'', one can show that an algorithm solving DLP in a generic group $G$ must perform $\mathcal{O}(\sqrt{p})$ operations, where $p$ is the largest prime divisor of $|G|$. Since many of the best-known algorithms for DLP have about this complexity we can conclude that they are more or less optimal.

Shoup's result requires treating groups in a very generic way, and one's mileage may vary if more information about the group elements is revealed. This is of particular interest for DLP on elliptic curves, which is the flavour of the problem most commonly encountered in modern communication. On the other hand, as far as we are aware, there are no prior results whereby more information about the group actually \textit{improves} the Shoup bound. We accomplish this in this work.

%It is well-known that, owing to Shor's algorithm, the security of cryptosystems based on the classical variant of the Discrete Logarithm Problem (DLP) are not secure against quantum adversaries \cite{shor1994algorithms}. In response to this threat to modern security, the cryptographic community has engaged in a period of intensive research to construct and cryptanalyse alternative cryptosystems, which are believed not to succumb to a quantum computer. The resulting field is known as ``post-quantum cryptography'', and is perhaps best exemplified by the NIST competition, which has now announced its first standards. The conclusion of this process (reference to some nist whitepaper or something) is more or less that we should seek computational problems from other areas of mathematics to those that traditional cryptosystems are built from - most notably, from the theory of lattices.

% Far less work, however, has been done on the computational difficulty of more general versions of DLP. In this direction, recall that Shor's famous algorithm shows that DLP in a finite cyclic group is quantum-easy -- but an instance of DLP in an arbitrary finite group, say, with respect to a base element $g$, is by definition a version of DLP in the finite cyclic group $\langle g \rangle$. In other words, generalising the problem to a finite non-abelian group is no help in terms of security.

Indeed, in this direction, an obvious generalization of the problem to study is DLP in an infinite group. 
% Intuitively, one would not expect a version of DLP with respect to an element of infinite order to succumb to Shor's method, simply because the period-finding routine at the heart of Shor's algorithm has no finite period to recover. 
Very little study exists on this subject - in this paper, we demonstrate that by moving to the infinite context, it is not only possible to obtain plausibly difficult instances of DLP, but instances of DLP that are as difficult as the evaluation of any function on the integers. 

More precisely, this paper provides a translation of the theory of one-way functions to an infinite-group-theoretic setting. We show that for any function $f:\nn\to\nn$, there exists an infinite abelian group in which the word problem is solvable in about the same time as the evaluation of $f^{-1}$, but the discrete logarithm problem is at least as hard as the evaluation of $f$. Loosely speaking, the ease of the word problem corresponds to the ability to easily compute group operations, so indeed, if $f$ is easy to invert and hard to compute, we have a tractable object in which DLP is as hard as desired. Up to the existence of one way functions, we generate controllably hard instances of DLP. 
~\\

Let $f: \nn \rightarrow \nn$ be a computable function. Let us denote by $T_f: \nn \rightarrow \nn$ the maximal number of steps that it takes to compute $f(m)$ for $1\leq m \leq \exp(n)$ on a fixed Turing machine that computes $f$. In other words, $T_f$ corresponds to the time complexity of computing $f$. Similarly, define $T_{f^{-1}}: \nn \rightarrow \nn$ is as $T_{f^{-1}}(n)$ is the maximum number of steps needed to verify whether or not $f(k)=l$ for any input $(k, l) \in \nn \times \nn$ satisfying $1\leq exp(kl) \leq n$ on a fixed Turing machine that solves the verification problem. In other words, $T_{f^{-1}}$ corresponds to the time complexity of the problem of verification of the correctness of the computation of $f$. We also say that an algorithmic problem $P_1$ is polynomial time equivalent to the problem $P_2$ if the time complexity for $P_1$ differs from the time complexity of $P_2$ at most by a polynomial factor. \\

Our main results obtained in this work are the following.
\begin{thm}[See Section \ref{section-the-embedding}]
\label{thm-dlp-from-f}
    For any computable function $f: \nn \rightarrow \nn$, there exists a two-generator group $G_f$ such that the discrete log problem in $G_f$can be solved in time $T_f$ and the word problem can be solved in time $ T_{f^{-1}},$ both up to polynomial factors. 
\end{thm}
The next theorem can be regarded as a corollary from Theorem \ref{thm-dlp-from-f}.
\begin{thm}[Corollary \ref{cor-np-complete-discrete-log}]
\label{thm-main-on-NP-DLP}

There exist two-generated groups with polynomial time word problem and NP-hard discrete log problem.
\end{thm}
\subsection{Related work.}
For a survey on the classical difficulty of solving DLP in a generic group, in the multiplicative group of a finite field, and in an elliptic curve, we recommend the survey article \cite{joux2014past}. 

There are some papers describing concrete instances of groups in which the generic lower bound is violated (or at least, make progress towards demonstrating this); that is, extra information about the group allows one to give an algorithm \textit{upper-bounded} asymptotically below the Shoup bound. Important examples of this type of work include \cite{lamacchia1991computation, barbulescu2014heuristic, mov}. Nevertheless, as far as we are aware, there are no examples in the literature of instances of concrete groups in which the lower bound for DLP is improved compared to the Shoup bound.

\subsection{The main premise  behind the group theoretical constructions in our paper} In the current paper we present several cryptographic protocols implemented in a setting of finitely generated infinite groups. To achieve this implementation, we start with (abelian) groups of countable rank, then describe constructions that allow us to carry essential algorithmic and cryptographic information that presentations of those abelian groups carry into the setting of finitely generated groups. The main conceptual premise is that for groups that are not finitely generated, most of the natural algorithmic properties are not intrinsic for groups themselves but intrinsic for particular enumerated presentations (i.e. enumerations of their elements), whereas for finitely generated groups those properties can be regarded as intrinsic properties of the groups themselves. An example of such an algorithmic property is the word problem in groups. 

\subsection{Quantum resistance.}

In the current landscape of the first standardised primitives believed to be quantum-resistant, a natural question to ask is the following: how resistant to quantum attack do we expect the DLP to be in the groups we have described? This is especially pertinent in light of the high-profile cryptanalytic surprises that emerged on the road to standardisation: in particular the recent attacks \cite{DBLP:journals/iacr/Beullens22,DBLP:conf/crypto/TaoPD21,DBLP:journals/iacr/BaenaBCPSV21} against the round-$3$ DSSs Rainbow \cite{DBLP:journals/iacr/Beullens22} and GEMSS \cite{casanova2017gemss}, as well as the cryptanalysis of round-$4$ isogeny-based KEM known as SIKE \cite{cryptoeprint:2022/975, cryptoeprint:2022/1026}. In other words, despite the completion of the standardisation process, there remains motivation to study a variety of candidate quantum-hard problems.

This work does not address the quantum resistance of DLP in the groups we introduce. On the one hand, the group we work in is somehow defined with respect to a group whose elements all have finite order, and so we would expect some Shor-like method period-finding to apply; on the other hand, the actual group $G_f$ in which the difficulty of DLP is to be analysed has elements of infinite order, and so the standard quantum techniques might not be expected to work. In order to settle the question, one would have to understand how the difficulty of the problem in these two groups is related from a quantum perspective, as well as understanding how quantum algorithms interact with symbolic algebra. These problems are considered rich enough to warrant further work. Indeed, during the preparation of this manuscript, a work giving detailed analysis of the hidden subgroup problem for infinite groups appeared on the arXiv \cite{kuperberg2025hidden}, so certainly there is interest in this kind of problem in the community.

 \section{Enumerated abelian groups and algorithmic problems on them}
\label{section-A_f}

Let 
\begin{align}
\label{eq-presentation-A}
    A = \langle a_1, a_2, \ldots \mid [a_i, a_j], i, j \in \nn; R \rangle
\end{align} 
be an abelian group with an enumerated set of generators $\{a_1, a_2, \ldots \}$, the set of relators $[a_i, a_j]:=a_ia_ja_i^{-1}a_j^{-1}$ and an additional set of relators denoted by $R$. In this work, we describe how algorithmic properties of the presentation \eqref{eq-presentation-A} can be carried out to a two-generated groups, making those properties intrinsic for the group itself and not just of its presentation. As main applications, we describe a family of finitely generated groups for which the Discrete Log Problem is polynomially equivalent to the computation of any function $f: \nn \rightarrow \nn$ fixed beforehand. This is carried out in Sections \ref{section-the-embedding}.

Next, in Section \ref{section-new-protocol}, we describe how this idea can be implemented in cryptographic setting by describing a general scheme of a key-exchange protocol that encapsulates properties of commutative actions of groups and one-way functions via encodings in two-generated infinite groups.

\subsection{The Discrete Log problem in countable abelian groups}
\label{subsection-basic-obs}
 Let $f: \nn \rightarrow \nn$ be a fixed computable map. In this subsection, we are going to describe a presentation of a countable abelian group for which the Discrete Log problem is polynomially equivalent to the computation of $f(n)$.

 By $T_f: \nn \rightarrow \nn$ we denote the time complexity of computing the values $f: \nn \rightarrow \nn$ on a fixed Turing machine. That is, $T_f(n)$ indicates the number of steps the machine takes to output $f(n)$ after input $n$.

 Similarly, by $T_{f^{-1}}: \nn \rightarrow \nn$ we denote the time complexity on a fixed Turing machine for the following problem: on input the pair $(n, m) \in \nn \times \nn$ is given. Verify whether or not $f(n)=m$.

%Let us denote by $p_n$ the $n$-th prime number and 
Define
 
\begin{align}
\label{eq-presentation-A-for-DiscreteLog}
    A_f: = \langle a_1, a_2, \ldots, b_1, b_2, 
    \ldots \mid [a_i, a_j], [a_i, b_j], [b_i, b_j], a_i^{{f(i)}}=b_i, i, j \in \nn \rangle.
\end{align} 

Note that each element of $A_f$ can be represented as a finite word of the form 
\[
w = a_{n_1}^{k_1} \ldots a_{n_t}^{k_t} b_{m_1}^{l_1} \ldots b_{m_s}^{l_s} \in \{ a_1^{\pm 1}, a_2^{\pm 1}, \ldots, b_1^{\pm 1}, b_2^{\pm 1}, \ldots \}^*.
\]

Also, note that each term $a_{n_i}^{k_i}$ can be encoded as a string of $0$'s and $1$'s of length $\Theta(\log(|n_ik_i|)$. Similarly, the term $b_{m_j}^{l_j}$ would require a string of size $\Theta(\log(|m_jl_j|)$. Thus, we define the length function
$${\ell}:\{ a_1^{\pm 1}, a_2^{\pm 1}, \ldots, b_1^{\pm 1}, b_2^{\pm 1}, \ldots \}^*\rightarrow \nn \cup \{0\} $$ as follows:
$$\ell(w):= \sum_{i=1}^{t} \log(|n_ik_i|) + \sum_{j=1}^{s} \log(|m_jl_j|) .$$

\begin{df}[Reduced word]
    We say that $w = a_{n_1}^{k_1} \ldots a_{n_t}^{k_t} \in \{ a_1^{\pm 1}, a_2^{\pm 1}, \ldots \}^*$ is in \emph{a reduced form} if $w$ is the empty word or $n_1, \ldots, n_t$ and $m_1, \ldots,  m_s$ are pair-wise not-equal, respectively, and $k_1, \ldots, k_t$, $l_1, \ldots, l_s$ are non-zero integers, and  it does not contain terms $a_{n}^{k}$ and $b_{n}^{l}$ such that $k = -l {f(n)}.$
\end{df}

\begin{lem}
    \label{lem-transformation-into-reduced-word}
    Any word $w  \in \{ a_1^{\pm 1}, a_2^{\pm 1}, \ldots, b_1^{\pm 1}, b_2^{\pm 1}, \ldots \}^*$ can be transformed into a reduced word $w'$ that represents the same element in $A_{f}$ as  $w$ in time $\calO((l(w))^2+l(w)) T_{f^{-1}}(l(w)) $.
\end{lem}
\begin{proof}
    Indeed, given $w  \in \{ a_1^{\pm 1}, a_2^{\pm 1}, \ldots, b_1^{\pm 1}, b_2^{\pm 1}, \ldots \}^*$, simply by transposition of terms $a_i^{\pm 1}, b_i^{\pm 1}$, first we can transform it into a word $w'=  a_{n_1}^{k_1} \ldots a_{n_t}^{k_t} b_{m_1}^{l_1} \ldots b_{m_s}^{l_s}$ that would satisfy the condition that  $n_1, \ldots, n_t$ and $m_1, \ldots,  m_s$ are pair-wise not-equal, respectively, and $k_1, \ldots, k_t$, $l_1, \ldots, l_s$ are non-zero integers. This can be done in time $\calO((l(w))^2$. 

    Therefore, without loss of generality, we can assume that $w$ satisfies  the condition that  $n_1, \ldots, n_t$ and $m_1, \ldots,  m_s$ are pair-wise not-equal, respectively, and $k_1, \ldots, k_t$, $l_1, \ldots, l_s$ are non-zero integers. In this case, to check if $w$ is in a reduced form, one can check for each term $a_{n_i}^{i}$, $i=1, \ldots, t$, if there is a term $b_{m_j}^{l_j}$, $j=1, \ldots, s$, such that $n_i = m_j$ and $n_i = -m_i {f(m_j)}$. This step can be done in time $\calO(l(w)) T_{f^{-1}}(l(w)) $. Therefore, the assertion of the lemma follows.
\end{proof}

\begin{cor}
\label{cor-wp-in-A_f}
    The word problem in $A_f$ is solvable in time $\calO(n^2 T_{f^{-1}}(n)) $, where $n$ is the length of the input word with respect to the length function $l$.
\end{cor}
\begin{proof}
    It is enough to notice that for $w \in \{ a_1^{\pm 1}, a_2^{\pm 1}, \ldots, b_1^{\pm 1}, b_2^{\pm 1}, \ldots \}^*$ given in a reduced form, it represents the trivial element of $A_f$ if and only if $w$ is the empty word. Then, the claim of the corollary follows from Lemma \ref{lem-transformation-into-reduced-word}.
\end{proof}
\begin{lem}
    \label{lem-complexity-disc-log-in-Af}
    The Discrete Log problem in $A_f$ with respect to the presentation \eqref{eq-presentation-A-for-DiscreteLog} of $A_f$ is polynomially equivalent to the computation of $f: \nn \rightarrow \nn$.
\end{lem}
\begin{proof}
    Indeed, let us consider an equation $u^x = v$, where $x \in \zz$ is the unknown, $u, v \in A_f$. Without loss of generality, assume that $u$ is given in a reduced form as
    $u =  a_{n_1}^{k_1} \ldots a_{n_t}^{k_t} b_{m_1}^{l_1} \ldots b_{m_s}^{l_s} \in \{ a_1^{\pm 1}, a_2^{\pm 1}, \ldots, b_1^{\pm 1}, b_2^{\pm 1}, \ldots \}^*$. Then, it follows from the definition of $A_f$ that if a reduced form of $v$ does not contain a generator $a_{n_i}$ or $b_{m_j}$ for $i=1,\ldots, t$, $j=1, \ldots, s$, or it contains a generator different from those, then $u^x = v$ has no solution. Therefore, without loss of generality, let us assume that $v$ is given in a reduced form as 
    $v =  a_{n_1}^{k'_1} \ldots a_{n_t}^{k'_t} b_{m_1}^{l'_1} \ldots b_{m_s}^{l'_s}.$

    By introducing dummy multipliers of the form $a_n^0$ or $b_m^0$, without loss of generality, we can assume $(n_1, \ldots, n_t) = (m_1, \ldots, m_s)$ and the exponents of the terms are non-necessarily non-zero, which is  the only deviation from the definition of the reduced form. 

    Under this convention, let 
    $u =  a_{n_1}^{k_1} \ldots a_{n_t}^{k_t} b_{n_1}^{l_1} \ldots b_{n_t}^{l_t}$ and 
    $v =  a_{n_1}^{k'_1} \ldots a_{n_t}^{k'_t} b_{n_1}^{l'_1} \ldots b_{n_t}^{l'_t}.$ Then, for each $i = 1, \ldots, t$, $x$ must satisfy the system of equations
\begin{align}
    \left\{
                      \begin{array}{ll}
                       \lfloor \frac{k_ix}{f(n_i)} \rfloor +l_i x = l_i' & \\
                       k_i x - \lfloor \frac{k_ix}{f(n_i)} \rfloor f(n_i) = k_i'  & \mbox{.}
                     \end{array}
                    \right.
  \end{align}
  Note that for each $i=1, \ldots, t$, this system can have at most one solution. Therefore, it is enough to find $x$ when $i=1$ then check if it satisfies the systems of equations for $i=2, \ldots, t$. Moreover, this $x$ can be found by formula $x = \frac{f(n_1) l_1' + k_1'}{f(n_1)l_1+k_1}.$ Therefore, assuming that $f(n_1), \ldots, f(n_t)$ are found, one can find $x$ or show that such $x$ does not exist in time $\calO(t\log(f(n_1)l_1l_1'k_1k_1')).$ This implies that finding $x$ or showing that $x$ does not exist can be done in time $\calO((\ell(u)+\ell(v))T_{\ell(u)+\ell(v)}).$

  On the other hand, consider the equation $a_n^x = b_n$. The solution of this equation is $x=f(n)$. This implies that the Discrete Log Problem in $A_f$ is at least as difficult as computing $f$. This completes the proof.

\end{proof}

 \section{Finitely generated groups with controllable Discrete Log Problem via group embeddings}
 \label{section-the-embedding}
 In this section, we are going to present an embedding construction of the abelian group $A_f$ from the previous section into a two-generated group $G_f$ such that the Discrete Log Problem for $G_f$ is polynomially equivalent to the computation of $f: \nn \rightarrow \nn$. The construction is an adaptation of a famous embedding construction of Neumann \cite{neumann_embedding_1960} (see also constructions from \cite{darbinyan_group_2015, arman_new}). Throughout this section,  let $f: \nn \rightarrow \nn$ be a fixed map as in the previous section, and $A_f$ is as in Section \ref{section-A_f}.

\subsection{The general idea}

Let $G_{f} = \langle x, y \rangle$ be a two generated group (the index $f$ does not play a role so far, but it's context will be clear later). Then,
for each word from $\{x^{\pm 1}, y^{\pm 1}\}^*$, given as
$w: = x^{k_0}y^{l_1}x^{k_1} \ldots y^{l_t} x^{k_t}y^{l_{t+1}}$, where $k_i, l_i \neq 0$ for $1\leq i \leq t$, we define its length $\ell(w)$, as
\[
\ell(w) = \log(|k_0|+1)+\log(|l_{t+1}|+1)+\sum_{i=1}^t \log(|k_il_i|+1),
\]
where for convenience we define $\log(0)=0$. One can think of $\ell(w)$ as the length (up to the linear equivalence)  of $w$ when encrypted via $0$'s and $1$'s. We refer to $\ell(w)$ simply as \emph{the length of the word $w$}.

Let $$\Psi: A_f \hookrightarrow G_{f} = \langle x, y \rangle $$
be an \emph{effective} embedding of the group $A_f$ into a two generated group $G_{f}$, where by effective embedding we mean that there is a \emph{polynomial time computable} map (with respect to the length of the input)
$$\phi: \nn \rightarrow \{x^{\pm 1}, y^{\pm 1}\}^*$$
such that for each $i \in \nn$, the image $\phi(i)$ is a word in $\{x^{\pm 1}, y^{\pm 1}\}^*$ that represents the element $\Psi(a_i)$ in $G_f$ and such that the length of $\phi(i)$ is bounded from above by $P(\log(i))$ for some fixed polynomial $P(n)$. For description of such embeddings, see, for example, \cite{darbinyan_group_2015, arman_new}.

Next, let us observe that the Discrete Log Problem in $G_f$ is at least as difficult as computation of $f: \nn \rightarrow \nn.$ 

Indeed, let us consider the equation $\Psi(a_i)^x = \Psi(b_i)$ for $i\in \nn$. Then, its solution is $x=f(i)$. Therefore, since $\Psi$ is an effective embedding, solving it is at least as difficult as computing $f(i)$.\\

In the next subsection we advance this insight by describing a concrete effective embedding of $A_f$ into a specific two generated group $G_f$.

\subsection{An effective embedding of $A_f$ into $G_f$.}

We start this section by fixing some standard group theoretical notations. Namely, for a group $G$ and $g, h \in G$, we denote the commutator of $g$ and $h$ by $[g, h]: = ghg^{-1}h^{-1}$. We also denote $g^h=hgh^{-1}$. For two subgroups $G_1, G_2 \leq G$, we denote by $[G_1, G_2]$ the subgroup of $G$ generated by commutator elements of the form $[g_1, g_2]$, where $g_1\in G_1$, $g_2 \in G_2$. We also denote $G':=[G, G]$.
 \begin{df}
The (unrestricted) wreath product $A \wr B$ is defined as a semidirect product $ A^B \rtimes B$, with the multiplication $(f_1b_1)(f_2 b_2) = f_1f_2^{b_1} b_1b_2$ for all $ f_1, f_2 \in A^B, ~b_1, b_2 \in B$, where $f_2^{b_1}: B \rightarrow A$ is defined as $f_2^{b_1}(b)=f_2(bb_1)$ for $b \in B$. In the definition, $A^B$ is the set of maps from $B$ to $A$ with a group structure defined by coordinate-wise multiplication.
\end{df}

Let 

\begin{align}
    \label{eq-presentation-of-A_f}
A_f: = \langle a_1, a_2, \ldots, b_1, b_2, 
    \ldots \mid [a_i, a_j], [a_i, b_j], [b_i, b_j], a_i^{{f(i)}}=b_i, i, j \in \nn \rangle.
\end{align}    
    be as in the previous section. In order to construct the group $G_f$, first, we will describe an embedding of the group $A_f$ into a subgroup $L$ of the unrestricted wreath product $A_f \wr \langle z \rangle $, where $\langle z \rangle$ is an infinite cyclic group generated by $z$. Define the functions $f_i: \langle z \rangle \rightarrow H $, $i=1, 2, \ldots$, as
\begin{align*}
   f_i(z^n)=  \left\{
                      \begin{array}{ll}
                       a_{\bar i}^n & \mbox{if $i$ is odd,}\\
                       b_{\bar i}^n & \mbox{if $i$ is even}
                     \end{array}
                    \right.
   \end{align*}
   where
\begin{align*}
  \bar i =  \left\{
                      \begin{array}{ll}
                       (i-1)/2 & \mbox{if $i$ is odd,}\\
                       i/2 & \mbox{if $i$ is even.}
                     \end{array}
                    \right.
   \end{align*}
   
Note that 
\begin{align}
\label{eq-evidence}
   [z, f_i](z^n)= ( z f_i z^{-1}f_i^{-1} )(z^n) =  ( z f_i z^{-1})(z^n) f_i^{-1}(z^n) = \left\{
                      \begin{array}{ll}
                        a_{\bar i}^{n+1} a_{\bar i}^{-n} = a_{\bar i} & \mbox{if $i$ is odd,}\\
                        b_{\bar i}^{n+1} b_{\bar i}^{-n} = b_{\bar i} & \mbox{if $i$ is even.}
                     \end{array}
                    \right.   
\end{align}

That is, $[z, f_i]$, regarded as a map from $\langle z \rangle$ to $A_f$, is the constant map equal to $a_{\bar i}$ or $b_{\bar i}$ depending on the parity of $i$.\\

Denote  $L=\langle z, f_1, f_2, \ldots \rangle < A_f \wr \langle z \rangle$. The next step is to embed the group $L$ into a two generator subgroup $G_f=\langle f, s \rangle$ of the full wreath product  $L \wr \langle s \rangle < A_f \wr\langle z \rangle \wr \langle s \rangle$, where $\langle s \rangle$ is another infinite cyclic group, generated by $s$. To that end, define the function $F: \langle s \rangle \rightarrow L $ as
\begin{align*}
   F(s^n)= \left\{
                      \begin{array}{ll}
                       z  & \mbox{if $n=0$ ,}\\
                       f_n  & \mbox{if $n>0$ ,}\\
                       1 & \mbox{otherwise.}
                     \end{array}
                    \right.
\end{align*}
Then, for $i\geq 0$, we have
\begin{align}
\label{eq-21}
   [F, F^{s^{i}}](s^n)=(FF^{s^{i}}F^{-1}F^{-s^{i}})(s^n)= \left\{
                      \begin{array}{ll}
                       [z, f_i]  & \mbox{if $n=0$ ,}\\
                       1 & \mbox{otherwise}
                     \end{array}
                    \right.
\end{align}
and for $i<0$, we have
\begin{align}
\label{eq-22}
   [F, F^{s^{i}}](s^n)=(FF^{s^{i}}F^{-1}F^{-s^{i}})(s^n)= \left\{
                      \begin{array}{ll}
                       [f_i, z]  & \mbox{if $n=i$ ,}\\
                       1 & \mbox{otherwise.}
                     \end{array}
                    \right.
\end{align}
Combining \eqref{eq-21} and \eqref{eq-22}, we get that for $i\geq 0$,
\begin{align}
\label{eq-positive-reduction-comm}
    [F, F^{s^i}]= [F, F^{s^{-i}}]^{-s^{i}}.
\end{align}
 Now, define $$G_f = \langle F, s \rangle <  A_f \wr\langle z \rangle \wr \langle s \rangle.$$

 Note that, by \eqref{eq-evidence} and \eqref{eq-21}, it directly follows that $A_f$ embeds into $G_f$. Namely, the map $\psi: a_i \mapsto  [F, F^{s^{2i+1}}], b_i \mapsto [F, F^{s^{2i}}]$ induces an embedding $\Psi$ of $A_f$ into $G_f$.

\begin{lem}
    \label{lem-Psi-is-effective}
    The embedding $\Psi: A_f \hookrightarrow G_f$ is effective. {Moreover, the lengths of the images of elements of $A_f$ are being preserved under $\Psi$ with respect to the length functions $\ell$.}
\end{lem}
\begin{proof}
    Indeed, for each $i, n \in \nn$, $\Psi: a_i^n \mapsto [F, F^{s^{2i}}]$ and $\Psi: b_i^n \mapsto [F, F^{s^{2i+1}}]$. Now, note that by a straightforwrad checking we have $[F, F^{s^{i}}]^n =  [F^n, F^{s^{i}}]$. Finally, for $[F^n, F^{s^{i}}]\in \{F^{\pm 1}, s^{\pm 1}\}^*$, $\ell_{A_f}(a_i^n)= \ell_{A_f}(b_i^n) = \log(ni) < \ell([F^n, F^{s^{i}}]) < 8 \log(ni), $ which implies the statement of the lemma.
\end{proof}
 After defining the group $G_f$, our next task is to show that the discrete log problem in $G_f$ is polynomially equivalent to the computation of $f: \nn \rightarrow \nn.$

 Note that by iteratively applying rewritings of the form $s^k F^l \rightarrow (F^{s^k})^l s^k$, every word formed by the alphabet $\{F^{\pm 1}, s^{\pm 1} \}^*$, in polynomial time (in fact, in linear time) can be transformed into a word of the form 
 \begin{align}
 \label{eq rewriting 1}
    w=(F^{s^{\alpha_1}})^{\beta_1}(F^{s^{\alpha_2}})^{\beta_2}\ldots (F^{s^{\alpha_n}})^{\beta_n} s^{\delta},
\end{align}
where  $\delta$, $\alpha_i$, $\beta_i$, $i=1, \ldots n$, are some integers. This means that every element of $G_f$ can be decomposed as in \eqref{eq rewriting 1}.

\begin{lem}
\label{lem-commutativity-of-commutators}
    For $\alpha, \beta, \gamma, k, l, m \in \zz$, we have $[[(F^{s^{\alpha}})^k, (F^{s^{\beta}})^l], (F^{s^{\gamma}})^m ] =1. $
\end{lem}
\begin{proof}
    From the above discussion, for any $n \in \zz$, $[(F^{s^{\alpha}})^k, (F^{s^{\beta}})^l] (s^n)$, as a map from $\langle z \rangle $ to $G_f$ is a constant map, taking values from the commutative group $A_f$. Therefore, $[(F^{s^{\alpha}})^k, (F^{s^{\beta}})^l] (s^n)$ commutes with $z$ in $A_f \wr \langle z \rangle$. This implies that $[[(F^{s^{\alpha}})^k, (F^{s^{\beta}})^l], (F^{s^{\gamma}})^m ] =1.$
\end{proof}
Let us denote $\calF = \langle F^{s^{\alpha}} \mid \alpha \in \zz \rangle < G_f$ and denote $\calF' = [\calF, \calF]<\calF$. As an immediate consequence of Lemma \ref{lem-commutativity-of-commutators}, we get
$$\calF' = \langle ([F, F^{s^{\beta}}]^{s^{\gamma}})^{l} \mid \beta, \gamma, l \in \zz \rangle.$$

As yet another immediate corollary from Lemma \ref{lem-commutativity-of-commutators}, we get the following.
\begin{cor}
\label{cor-F'-is-center}
    $\calF' = Z(\calF)$, that is $\calF'$ coincides with the center of $\calF$.
\end{cor}
\begin{lem}
\label{lem-on-reduced-form}
    Every word $w \in \{F^{\pm 1}, s^{\pm 1} \}^*$ can be transformed in linear time into a word that represents the same element in $G_f$ as $w$ and such that it is of the form
    \begin{align}
        \label{eq-reduced-form}
    \prod_{i=1}^{t_1} (F^{s^{\alpha_i}})^{k_i} \prod_{i=1}^{t_2} ([F, F^{s^{\beta_i}}]^{s^{\gamma_i}})^{l_i} s^y,
    \end{align} 
    where $\alpha_1>\alpha_2 > \ldots > \alpha_{t_1}$, $(\beta_1, \gamma_1)>_{lex}(\beta_2, \gamma_2)>_{lex}\ldots>_{lex}(\beta_{t_2}, \gamma_{t_2})$, $\beta_i>0$, $i=1, \ldots, t_2$, $k_i\neq 0$ if $t_1\neq 0$ and $l_i\neq 0$ if $t_2\neq 0$.
    Moreover, the parameters $\alpha_1, \ldots, \alpha_{t_1}$, $k_1, \ldots, k_{t_1}$, and $y$ are uniquely defined for any given element of $G_f$. 
\end{lem}
\begin{proof}
    Indeed, one can start with a presentation of the form \eqref{eq rewriting 1}, then iteratively apply rewritings of the form $ab \rightarrow [a, b]ba$ in order to collect the terms $(F^{s^{\alpha_i}})^{\beta_i}$, $(F^{s^{\alpha_j}})^{\beta_j}$ for $\alpha_i = \alpha_j$ together. Because of Lemma \ref{lem-commutativity-of-commutators}, the commutator terms that will occur in the process of this rewriting are of the form $([F, F^{s^{\beta}}]^{s^{\gamma}})^{l}$. Now, because of the equation \eqref{eq-positive-reduction-comm}, without loss of generality, we can always assume that $\beta>0$. Also , as $\calF'=Z(\calF)$ due to Corollary \ref{cor-F'-is-center}, we can combine the terms of the form $([F, F^{s^{\beta}}]^{s^{\gamma}})^{l}$ in a way to obtain the decomposition \eqref{eq-reduced-form}. 

    Now, let us show the uniqueness of the parameters  $\alpha_1, \ldots, \alpha_{t_1}$, $k_1, \ldots, k_{t_1}$ and $y$ in the decomposition \eqref{eq-reduced-form}. The uniqueness of $y$ follows immediately from the construction (basic properties of semi-direct products) and the fact that $\langle s \rangle$ is an infinite cyclic group. To show the uniqueness of the other parameters, let us denote $\tilde F:=\prod_{i=1}^{t_1} (F^{s^{\alpha_i}})^{k_i} \prod_{i=1}^{t_2} ([F, F^{s^{\beta_i}}]^{s^{\gamma_i}})^{l_i}$ and regard it as a map from $\langle s \rangle$ to $L=\langle z, f_1, f_2, \ldots \rangle< A_f \wr \langle z \rangle.$ Then, note that, for $i=1, \ldots, t$, $\tilde F(s^{-\alpha_i})$ is of the form $\tilde f z^{k_i}$, where $\tilde f  \in \langle f_1, f_2, \ldots \rangle$; whereas, if $\alpha \in \zz\setminus \{\alpha_1, \ldots, \alpha_t \}$, then $\tilde F(s^{-\alpha}) \in \langle f_1, f_2, \ldots \rangle$. Also, since $L$ splits as $L= \langle  f_1, f_2, \ldots \rangle^{\langle z \rangle} \rtimes \langle z \rangle$, we get that in any decomposition of an element of $G_f$ in the form \eqref{eq-reduced-form}, the parameters $\alpha_1, \ldots, \alpha_{t_1}$, $k_1, \ldots, k_{t_1}$ are uniquely defined as well. 
   
\end{proof}

\begin{cor}
\label{cor-WP-in-G_f}
\label{cor-complexity-WP-in-Af}
    The word problem in $G_f$ in linear time reduces to the word problem in $A_f$ with respect to the presentation \eqref{eq-presentation-of-A_f}. In particular, the word problem in $G_f$, up to a polynomial factor, can be computed in time $T_{f^{-1}}(n)$, where $n$ is the length of the input with respect to $\ell$.
\end{cor}
\begin{proof}
    Let $g \in G_f$ be given as a word in $\{F^{\pm 1}, s^{\pm 1}\}^*$.  By Lemma \ref{lem-on-reduced-form}, in linear time, a canonical decomposition of the form \eqref{eq-reduced-form} can be found for $g$. Again, by Lemma \ref{lem-on-reduced-form}, if $g=1$, then in its canonical decomposition the parameters $\alpha_1, \ldots, \alpha_{t_1}$, $k_1, \ldots, k_{t_1}$ and $y$ would all be equal to $0$. Therefore, assuming that $\alpha_1, \ldots, \alpha_{t_1}$, $k_1, \ldots, k_{t_1}$ and $y$ are all equal to $0$, checking whether or not $g=1$ would be equivalent to checking whether or not 
    \begin{align}
        \label{eq-reduction-to-A_f}
    \prod_{i=1}^{t_2} ([F, F^{s^{\beta_i}}]^{s^{\gamma_i}})^{l_i}
    \end{align}
    is equal to $1$ in $G_f$, where the notation comes from \eqref{eq-reduced-form}. To check this, one can regroup the terms $[F, F^{s^{\beta_i}}]^{s^{\gamma_i}}$ according to the values $\gamma_i$, that is, for each $\gamma \in \{\gamma_1, \ldots, \gamma_{t_2} \}$, one can consider the subproduct of terms $[F, F^{s^{\beta_i}}]^{s^{\gamma_i}}$ for which $\gamma_i=\gamma$. Notice that, by \eqref{eq-21} and \eqref{eq-evidence}, checking whether or not \eqref{eq-reduction-to-A_f} represents $1$ in $G_f$ would be the same as to check whether or not the element in $A_f$ that corresponds to those subproducts via the correspondence $[F, F^{s^{\beta_i}}]^{s^{\gamma_i}} \leftrightarrow a_{\beta_i}^{l_i} \in A_f$ (for each  $\gamma \in \{\gamma_1, \ldots, \gamma_{t_2} \}$) represent the trivial element of $A_f$. Thus the word problem in $G_f$ in linear time reduces to the word problem in $A_f$.  
\end{proof}

\subsection{The Discrete Log Problem in $G_f$}
In this subsection we show that the discrete log problem in $G_f$ is polynomially equivalent to the computation of $f$.

\begin{df}
    For any $g \in G_{f}$, a decomposition of $g$ in the form \eqref{eq-reduced-form}, where the parameters in the form satisfy the conditions from Lemma \ref{lem-on-reduced-form}, we call a \textbf{canonical decomposition} or \textbf{canonical form} of $g$.
\end{df}

Let $g, h \in G_f$. Let us consider the equation
\begin{align}
    \label{eq-log-eq}
    g^x=h, \text{~where $x \in \zz$ is the unknown.}
\end{align}
\
Notation: For $g \in G_f$ that has a canonical decomposition \eqref{eq-reduced-form}, we denote $\pi_s(g): = y$. 

Note that for $g_1, g_2 \in G_f$, $\pi_s(g_1g_2) = \pi_s(g_1)+\pi_s(g_2).$ Therefore, if in \eqref{eq-log-eq} $\pi_s(h)\neq 0$, then it would imply that $x\pi_s(g) = \pi_s(h)$, that is $x=\pi_s(h)/\pi_s(g).$ (In particular, it would mean that $\pi_s(h)$ is a multiple of $\pi_s(g)$.) Therefore, \eqref{eq-log-eq} would have a solution (namely, $x=\pi_s(h)/\pi_s(g)$) if and only if
$g^{\pi_s(h)/\pi_s(g)}h^{-1}=_{G_f} 1.$ Since, by Corollary \ref{cor-complexity-WP-in-Af}, the word problem in $G_f$ is polynomially equivalent to the computation of $f$, we conclude that if $\pi_s(g)\neq 0$ or $\pi_s(h)\neq 0$, then solving \eqref{eq-log-eq} is at most as difficult as computing $f$. Thus we are left with the case when $\pi_s(g)=\pi_s(h)=0$.
~\\
\noindent
\begin{notation}
For $g$ with canonical decomposition as in \eqref{eq-reduced-form}, we denote $\deg_{\alpha_i}(g):=k_i$ for $i=1, 2, \ldots, t_1$. For $\alpha \in \nn \setminus \{\alpha_1, \ldots, \alpha_{t_1}\}$, we define $ \deg_{\alpha}(g)=0$.
\end{notation}
Note that
$$\text{if $\pi_s(g)=0$, then $\deg_{\alpha}(gh)=\deg_{\alpha}(g)+deg_{\alpha}(h)$} \text{~for all~ $\alpha \in \nn$}.$$
Therefore, assuming $\pi_s(g)=0$, we get $x\deg_{\alpha}(g)=\deg_{\alpha}(h)$ for $\alpha \in \nn$. This means that either $\deg_{\alpha_i}(g)=\deg_{\alpha_i}(h)=0$ for $i=1, 2, \ldots, t_1$, equivalently, $g, h \in \calF'$, or there is only one possibility for $x$, namely $x = \deg_{\alpha_i}(h)/\deg_{\alpha_i}(g)$, where $1\leq i \leq t_1$ is such that $\deg_{\alpha_i}(g) \neq 0$. The last case can be dealt analogously to the above discussion, namely, it is at most as difficult as computing $f$ to check if $g^{\deg(h)/\deg(g)}=h$.

Thus we are left with only one final case, which is $g, h \in \calF'.$ 

As in the case of the word problem in $G_f$ (see Corollary \ref{cor-WP-in-G_f}), we can regroup the factors of the form $([F, F^{s^{\beta}}]^{s^{\gamma}})^{l}$, $\beta>0$, in canonical decompositions of $g$ and $h$, according to the values of $\gamma$. Then, for each $\gamma$, the corresponding product of elements of the form $([F, F^{s^{\beta}}]^{s^{\gamma}})^{l}$ on the left hand side, raised to its $x$-th power, must be equal to the corresponding product of element on the right hand side.  The crucial thing to note is 
that, for a fixed $\gamma>0$, 

\[
[F, F^{s^{\beta}}]^{s^{\gamma}}\leftrightarrow  \left\{
                      \begin{array}{ll}
                       a_{\bar \beta}  & \mbox{if $\beta$ is odd }\\
                       b_{\bar \beta} & \mbox{if $\beta$ is even }
                     \end{array}
                    \right.
\]
induces an isomorphism between the group $\langle [F, F^{s^{\beta}}]^{s^{\gamma}}, \beta \in \nn $ and $A_f$, which directly follows from the equations \eqref{eq-evidence} and \eqref{eq-21}, where 
\begin{align*}
  \bar \beta =  \left\{
                      \begin{array}{ll}
                       (\beta-1)/2 & \mbox{if $\beta$ is odd,}\\
                       \beta/2 & \mbox{if $\beta$ is even.}
                     \end{array}
                    \right.
\end{align*}
    Therefore, in this case, the discrete log problem (in linear time) reduces to the discrete log problem in $A_f$, which, with respect to the length function $\ell$, is linear time equivalent to the computation of $f: \nn \rightarrow \nn$ according to the Lemma \ref{lem-complexity-disc-log-in-Af}.

\begin{conclusion}
\label{conclusion-1}
    The discrete log problem in $G_f$ is linear time equivalent to the computation of $f$.
\end{conclusion}

\begin{cor}[Theorem \ref{thm-main-on-NP-DLP}]
\label{cor-np-complete-discrete-log}
There exist two-generated groups with polynomial time word problem and NP-hard discrete log problem.
\end{cor}
\begin{proof}
    Indeed, let us assume that $f: \nn \rightarrow \nn$ is an NP-hard function and such that $f(n)$ serves as its polynomial-time certificate (i.e. witness), meaning that for $(n, m) \in \nn \times \nn$, one in polynomial time can verify whether or not $f(n)=m$. Then, according to Corollaries \ref{cor-complexity-WP-in-Af}, \ref{cor-wp-in-A_f}, the word problem in $G_f$ will be solvable in polynomial time, while by Conclusion \ref{conclusion-1}, the discrete log problem in $G_f$ is NP-hard.
\end{proof}
%%%%%%%%%%%%%%%%%%%%%%%%%%%%%%%%%%%%%%%%%%%%%%%%%%%%%%%%%%%%%% 

%%%%%%%%%%%%%%%%%%%%%%%%%%%%%%%%%%%
\section{A new cryptographic protocol}
\label{section-new-protocol}
General idea:

We start with two functions $f, g: \nn \times \nn \rightarrow \nn$ such that $f$ is a one-way function. We need the following commutative relation to hold between $f$ and $g$: for every $n, p, q \in \nn$,
\begin{align}
    \label{eq-for-new-chemes}
    g(f(n, p), q) = g(f(n, q), p).
\end{align}
Next, we consider three countable abelian groups with enumerated sets of generators:
\[
A = \langle a_1, a_2, \ldots \rangle,
\]
\[
B = \langle b_1, b_2, \ldots \rangle,
\]
\[
C = \langle c_1, c_2, \ldots \rangle.
\]
such that $A \leq B \leq C$.

We define additional relations between $A$, $B$ and $C$ as follows: for $m,n, p, q \in \nn$,
\[
    a_n^p = b_{f(n, p)}
\]
and
\[
    b_m^q = c_{g(m, q)}.
\]
\begin{rem}
    Note that in order to make sure that the above defined relators between the elements of $A$, $B$, and $C$ do not make generating elements trivial in the group $A\times B \times C$, we need to put extra conditions on the maps $f, g: \nn \times \nn \rightarrow \nn$. For example, requiring them to be injective would work. However, it is not critically important if elements of the groups are represented by their symbolic forms as finite words formed by a generating alphabet. Although, it might  become important once $A, B, C$ are replaced by more concrete realizations, say, by numerical matrices.
\end{rem}
\subsection{A cryptographic key sharing scheme}
~\\
\underline{Public data: The functions $f$ and $g$ are  public, as well as $n\in \nn$} \\
\underline{Private key of Alice: $p\in \nn$} \\
\underline{Private key of Bob: $q\in \nn$} \\

The general idea of this cryptographic scheme is that Alice applies her private key $p$ for $a_n \mapsto a_n^p = b_{f(n, p)}$. Then, Bob accesses $b_{f(n, p)}$ and applies his private key $q$ for $b_{f(n, p)} \mapsto b_{f(n, p)}^q = c_{g(m, q)}$. Moreover, Bob, again using his private key $q$, makes public the following $a_n \mapsto a_n^q = b_{f(n, q)}$

\underline{The shared key between Alice and Bob:} $c_{g(m, q)}$ is the shared key between Alice and Bob. The reason why this key can be secretly shared lies in the Equation \eqref{eq-for-new-chemes}, as both $b_{f(n, p)}$ and $b_{f(n, q)}$ are publicly available. Moreover, since $f$ is a one-way function, the shared key will stay private.

\subsection{Adaptation in the setting of our group theoretical construction:} Note that in our construction, we encoded the generating elements of $\calC:=\langle a_1, a_2, \ldots, b_1, b_2, \ldots, c_1, c_2, \ldots \rangle$ via generating elements $F$ and $s$ of the group $G_{\calC}$, where $G_{\calC}$ is defined as $G_f$ from Section \ref{section-the-embedding} but with respect to $\calA$. This allows us to adapt this scheme within the group-theoretical platform provided by $G_{\calC}$. (Note that $\calC$ as a group is the same as $C$, but we emphasize that all elements $a_i, b_i, c_i$ must be accounted for the definition of $G_{\calC}$.)

\subsection{Possible candidates for $f, g: \nn \times \nn \rightarrow \nn$}
~\\
\underline{A general approach via group actions:} In this approach, we will assume that $f$ and $g$ in fact coincide. Let $\Gamma = \{1= g_0, g_1, g_2, \ldots \}$ be some countable group with computably enumerated elements $1=g_0, g_1, \ldots$. Let $X=\{x_1, x_2, \ldots \}$ be a countable set with computably enumerated elements such that $\Gamma$ acts on $X$ via \emph{commutative group action.} Let $\tilde f: X \rightarrow X$ be a one-way function that is a permutation of $X$

In this setting, the value of $f(n, p)$ can be defined as the enumeration of the element $g_p \cdot \tilde f(x_n)$. Note that, in this setting, since $f$ and $g$ coincide, Equation \eqref{eq-for-new-chemes} turns into
\[
f(f(n, p), q) = f(f(n, q), p),
\]
which is equivalent to
\[
(g_q \cdot g_p) \tilde f(x_n)= g_q \cdot (g_p \cdot \tilde f(x_n)) = g_q \cdot (g_p \cdot \tilde f(x_n)) = (g_p \cdot g_q) \tilde f(x_n),
\]
which holds because by our assumption $G \curvearrowright X$ is commutative.

\section{Conclusions}

We have also initiated the study of the cryptographic utility of this type of construction, touching on its post-quantum viability, and giving an example of a cryptopgraphic scheme. Nevertheless, our hope is that this work will spark further work investigating each of these two avenues in greater detail.

\section*{Acknowledgements}
The authors acknowledge the support from
the Institut Henri Poincaré (UAR 839 CNRS-Sorbonne Université) and LabEx CARMIN (ANR-10-LABX-59-01). 
DK and CB conducted this work partially with the support of ONR Grant 62909-24-1-2002. AD and DK thank Institut des Hautes \'Etudes Scientifiques - IHES for providing stimulating environment while this project was partially done. They were both supported by the CARMIN fellowship during the completion of this project. All authors acknowledge the support of IHP during the trimester on Post-quantum Algebraic Cryptography in the Fall 2024.

\mbox{}

 \addtocontents{toc}{\setcounter{tocdepth}{-10}}
 
\bibliographystyle{alpha}

\bibliography{eilos}

\end{document}